\documentclass[11pt,onecolumn,a4paper,oneside]{article}
\usepackage[utf8]{inputenc}
\usepackage[ngerman, english]{babel}
\usepackage[T1]{fontenc}
\usepackage{amsmath}
\usepackage{amssymb}
\usepackage{latexsym}
\usepackage{amsthm}
\usepackage{hyperref}

\theoremstyle{plain}
\newtheorem{prop}{Proposition}[section]
\newtheorem{theo}[prop]{Theorem}
\newtheorem{cor}[prop]{Corollary}
\newtheorem{lem}[prop]{Lemma}
\newtheorem{defn}[prop]{Definition}
\newtheorem{rem}[prop]{Remark}

\begin{document}
\begin{center}
\Large An Open Mapping Theorem for rings which have a zero sequence of units\\
\vspace{0.5cm}
\normalsize Timo Henkel
\footnote{timo.henkel@gmx.net}
\vspace{0.2cm}
\end{center} 
\noindent \textbf{Abstract}. This paper provides an Open Mapping Theorem for topological modules over rings that have a zero sequence consisting of units. As an application it is shown that there is a unique complete and metrisable topolo\-gy on finitely generated modules over certain topological rings (e.g. over complete noetherian Hausdorff Tate rings).
\section*{Introduction}
Let $f$: $X\rightarrow Y$ be a linear and continuous map of two Banach spaces. According to the Open Mapping Theorem (also known as Banach-Schauder Theorem) this map is surjective if and only if it is open. A more general variant of this theorem, in which $X$ and $Y$ are completely metrisable topological vector spaces over a non discrete valued division ring, holds as well (c.f. [BTVS] No.3,  §3, Theorem 1). In this paper we generalize this theorem to completely metrisable topological Hausdorff modules over certain topological rings.  It turns out that we do not even need a valuation on the considered ring. We just have to ensure that there are sufficient many units sufficiently close to $0$. More precisely we will see that it is sufficient for the considered topological ring to have a sequence of units converging to $0$. For example Tate rings, which are of importance in the theory of adic spaces, have this property. Thus the main result of this paper is an Open Mapping Theorem for such rings:
\begin{theo} (c.f. (\ref{setting}))
Let $R$ be a topological ring that has a sequence of units converging to $0$. Let $M$ and $N$ be topological Hausdorff $R$-modules that have countable fundamental systems of open neighborhoods of $0$. Let $M$ be complete, and let $u$: $M\rightarrow N$ be a continuous $R$-linear map. Then the following properties are equivalent:
\begin{enumerate}
\item[1)] $N$ is complete and $u$ is surjective.
\item[1a)] $N$ is complete and $u(M)$ is open in $N$.
\item[2)] $u(M)\subseteq N$ is not meagre.
\item[3)] For all neighbourhoods $V $of $0$ in $M$, $\overline{u(V)}$ is a neighbourhood of $0$ in $N$.
\item[4)] $u$ is open.
\end{enumerate}
\end{theo}
\noindent The proof of this theorem is a  generalization of the proof of the Open Mapping Theorem for topological vector spaces as it is given in [BTVS] No.3,  §3. Although the essence of this theorem surely is the implication of "1) $\Rightarrow$ 4)", the other implications are of interest as well.\\

\noindent This result provides a criterion for a linear map between finitely generated topological modules over such a topological ring to be continuous.

\begin{theo}(c.f. (\ref{lemma cont map}))
Let $R$ be a complete topological Hausdorff ring that has a sequence of units converging to $0$ and a countable fundamental system of open neighbourhoods of $0$. Let $M$ and $N$ be topological $R$-modules such that $M$ is Hausdorff, complete, finitely generated and has a countable fundamental system of open neighbourhoods of $0$. Then each $R$-linear map $f$: $M\rightarrow N$ is continuous.
\end{theo}

\medskip
\noindent As mentioned above, those results can be applied to Tate rings, which are defined in the last section of this paper (c.f (\ref{defn. Tate Ring})). Tate rings have a topologi\-cally nilpotent unit and thus particularly have a sequence of units converging to $0$. Using this theorems we will get the following result for complete noetherian Hausdorff Tate rings: 

\begin{theo}(c.f. (\ref{unique r mod strc}))
Let $R$ be a complete non-archimedean noetherian topological Hausdorff ring that has a countable fundamental system of neighbourhoods of $0$. Assume that there is a zero sequence of units of $R$ and that the set of all topologically nilpotent elements of $R$ is a neighbourhood of $0$ (e.g. $R$ is a complete noetherian Hausdorff Tate ring).
Let $M$ be a finitely generated $R$-module.\\ Then there exists a unique $R$-module topology on $M$ such that $M$ is Hausdorff, complete and has a countable fundamental system of open neighbourhoods of $0$ (i.e. $M$ is metrisable and complete).
\end{theo}

\noindent This paper is divided into two sections.  
In the first section we will prove Theorem 0.1 and Theorem 0.2 as a corollary. The second section gives the definition of Tate rings and the proof of Theorem 0.3.\\

\noindent \textbf{Notations and assumptions}
 \begin{enumerate}
\item Let all rings be with 1. All rings in the first section are not necessarily commutative, whereas all rings in the second section shall have this property.
\item For a ring $R$ an $R$-module is a left $R$-module.
\item For a subspace $Y$ of a topological space $X$, $\overline{Y} \subseteq X$ denotes the topological closure of $Y$ in $X$.
\item For a ring $R$, $R$* denotes the group of units of $R$.
\item For a metric space $X$ with metric $d$, $B_r(x_0)$ denotes the closed ball with centre $x_0$ and radius $r$, i.e. $B_r(x_0)={ \lbrace x\in X \ | \ d(x,x_0)\leq r  \rbrace }$.
\end{enumerate}

\noindent \textbf{Acknowledgements}.
I am deeply grateful to Torsten Wedhorn who engaged my interest for this topic by illustrating its relevance and applications. I also appreciate his patience and helpful advice he gave to me in countless discussions.

\section{Open Mapping Theorem for rings with a unit zero sequence}

After giving some remarks on topological groups and Baire spaces, we will state and prove an Open Mapping Theorem for topological rings in this section.

\subsection{Definitions and remarks on topological groups and Baire spaces}

\begin{defn}\label{complete}
\emph{Let $G$ be a topological group with neutral element $e$. \begin{enumerate}
\item A \emph{Cauchy sequence} $(x_n)_{n\in \mathbb{N}}$ in $G$ is a sequence in $G$ such that for all neighbourhoods $V$ of $e$ there is an $N\in \mathbb{N}$ with $x_m\cdot x_k^{-1} \in V \ \forall \ m,k\geq N$.
\item If $G$ has a countable fundamental system of open neighbourhoods of $e$, it is said to be \emph{complete} if every Cauchy sequence in $G$ converges in $G$.
\item  A topological Hausdorff module (resp. a topological Hausdorff ring) that has a countable fundamental system of open neighbourhoods of the neutral element of its additive group  is said to be \emph{complete} if its additive group is complete.
\end{enumerate}}
\end{defn}

\noindent In general one defines completeness for uniform spaces by the notion of Cauchy filters. Since in the metrisable case this is equivalent to work with Cauchy sequences (c.f. [BGT] IX, §2.6, Prop. 9), the definition given above is sufficient for this paper . \\

\noindent The following proposition shows that a complete topological group, that has a countable fundamental system of open neighbourhoods of its neutral element, is completely metrisable. 
\medskip

\begin{prop}\label{metrisable}
A topological group $G$ with neutral element $e$ is metrisable if and only if it is Hausdorff and has a countable fundamental system of open neighbourhoods of $e$. \\
If $G$ is metrisable one can find a metric $d$ on $G$ that is right-invariant (i.e. $d(x,y)=d(x\cdot z, y\cdot z)$ for all $x,y,z \in G$)  (resp. left-invariant) and induces the topology $G$ is endowed with. 
\end{prop}

\begin{proof}
\ensuremath{}  [BGT] IX, §3.1, Proposition 1 and Proposition 2.
\end{proof}

\medskip

\begin{cor}\label{complete together}
Let $G$ be a complete topological Hausdorff group that has a countable fundamental system of open neighbourhoods of its neutral element. Then $G$ is completely metrisable. 
\end{cor}

\begin{defn}\label{baire spaces} (Baire spaces) \\
\emph{A topological space $X$ is said to be a \emph{Baire space} if every nonempty open subset of $X$ is not meagre in $X$.}
\end{defn}

\begin{prop}\label{baire theo} (Baire's Theorem)\\
Every completely metrisable topological space is a Baire space.
\end{prop}
\begin{proof}
\ensuremath{} [BGT] IX, §5.3, Theorem 1.
\end{proof}

\medskip

\noindent The following subsections give the proof of the main result of this work:

\begin{theo}\label{setting}
Let $R$ be a topological ring that has a sequence of units converging to $0$. Let $M$ and $N$ be topological Hausdorff $R$-modules that have countable fundamental systems of open neighborhoods of $0$. Let $M$ be complete, and let $u$: $M\rightarrow N$ be a continuous $R$-linear map. Then the following properties are equivalent:
\begin{enumerate}
\item[1)] $N$ is complete and $u$ is surjective.
\item[1a)] $N$ is complete and $u(M)$ is open in $N$.
\item[2)] $u(M)\subseteq N$ is not meagre.
\item[3)] For all neighbourhoods $V $of $0$ in $M$, $\overline{u(V)}$ is a neighbourhood of $0$ in $N$.
\item[4)] $u$ is open.
\end{enumerate}
\end{theo}

\noindent The proof will follow and generalize arguments of [BTVS] I, §3. We will prove each implication of the theorem within each of the following subsections, where $1) \Leftrightarrow 1a)$ is a consequence of (\ref{M=N}). Since some implications can be proven in a more general variant, we will not always operate in the exact setting above. We have already shown $1) \Rightarrow 2)$. Hence we start with:

\subsection{2) implies 3)}

\begin{lem}\label{conv.in.top.mod}
Let $R$ be a topological ring and $(r_i)_{i\in I}$ an arbitrary family of elements of $R$ such that for each neighbourhood $U$ of $0$ there is an $i\in I$ with $r_i \in U$ (e.g. $I=\mathbb{N}$ and  $(r_i)_{i\in I}$ is a sequence converging to $0$). Let $M$ be a topological $R$-module. Then for every $m\in M$ and for every neighbourhood $V$ of $0$ in $M$ there is an $i\in I$ with $r_i\cdot m\in V$.
\end{lem}
\begin{proof}
Let $\pi$: $R \times M \rightarrow M$ be the continuous scalar multiplication of the $R$-module $M$. 
For all $m\in M$ there are $0\in U_1 \subseteq R$ open and $m\in U_2\subseteq M$ open with $ U_1 \times U_2 \subseteq \pi^{-1}(V)$. By assumption there is an $j \in I$ with $r_j\in U_1$. Therefore $(r_j,m)\in \pi^{-1}(V)$ i.e. $r_j\cdot m \in V$.
\end{proof}
\medskip

\begin{rem}\label{add of interp}
\emph{Let $G$ be a topological group. Let $x_1, x_2 \in G$ such that $V_i \subseteq G$ is a neighbourhood of $x_i$ for $i=1,2$. Then $V_1\cdot V_2$ is a neighbourhood of $x_1\cdot x_2$.}
\end{rem}

\begin{prop}\label{2->3}
$2)\Rightarrow 3)$: Let $R$ be a topological ring that has a sequence of units  $(r_n)_{n\in \mathbb{N}}$ converging to $0$. Let $M$ and $N$ be topological $R$-modules, and let $u$: $M\rightarrow N$ be a linear map such that $u(M)\subseteq N$ is not meagre. Then for every neighbourhood $V\subseteq M$ of $0$ the set $\overline{u(V)}$ is a neighbourhood of $0$ in $N$.
\end{prop}
\begin{proof}
Let $0 \in W \subseteq M$ be open with $W+W\subseteq V$, and assume that $-W=W$. By (\ref{conv.in.top.mod}) one has: 
\begin{equation*} \forall m \in M \ \exists \ l_m \in \mathbb{N} : r_{l_m}\cdot m \in W.
\end{equation*} 
So $m\in (r_{l_m}^{-1}W)$ and
\begin{equation*}
M=\displaystyle \bigcup_{n\in \mathbb{N}} (r_n^{-1} W).
\end{equation*}
Since $u$ is linear,
\begin{equation*}
u(M)=u(\displaystyle \bigcup_{n\in \mathbb{N}} (r_n^{-1} W))=\displaystyle \bigcup_{n\in \mathbb{N}} (r_n^{-1} u(W)) \subseteq \displaystyle \bigcup_{n\in \mathbb{N}} (r_n^{-1} \overline{u(W)}).
\end{equation*}
As $r_n \in R^*$, scalar multiplication with $r_n$ is a homeomorphism in $N$ $\forall \ n \in \mathbb{N}$. Thus  $r_n^{-1} \overline{ u(W)} \subseteq N$ is closed $ \forall n \in \mathbb{N}$. Since $u(M)$ is not meagre, it is not contained in a countable union of closed sets each of which has no interior point.
\noindent Thus $\exists \ l \in \mathbb{N} $ such that $r_l^{-1} \overline{u(W)}$ posses an interior point. But $r_l^{-1} \overline{u(W)}$ is homeomorphic to $\overline{u(W)}$ by scalar multiplication with $r_l$ and hence $ \overline{u(W)}$ posses an interior point $y$. Then $-y$ certainly is an interior point of $-\overline{u(W)} = \overline{-u(W)}=\overline{u(-W)}=\overline{u(W)}$. By (\ref{add of interp}) we get that $0=y-y$ is an interior point of $\overline{u(W)}+\overline{u(W)}$. But since addition in $N$ is continuous, $\overline{u(W)}+\overline{u(W)}$ is a subset of $\overline{u(W)+u(W)}$, which in turn is a subset of $\overline{u(V)}$. Hence $\overline{u(V)}$ is an neighbourhood of $0$ in $N$.
\end{proof} 
\medskip

\subsection{3) implies 4)}

\begin{prop}\label{closure of balls}
Let $M$ and $N$ be metric spaces.  Assume that $M$ is complete and consider a continuous map $u$: $M\rightarrow N$ having the following property: 
\begin{center}
For all $r\in \mathbb{R}^{>0}$ there exists  $\rho(r) \in \mathbb{R}^{>0}$ such that $B_{\rho(r)}(u(x)) \subseteq \overline{u(B_r(x))} \ \forall x\in M$.\\
\end{center}
Then one has $B_{\rho(r)}(u(x)) \subseteq u(B_a(x)) \ \forall \ a > r, \ x\in M.$
\end{prop}
\begin{proof}
\ensuremath{}  [BTVS] §3, Lemma 2.
\end{proof} 

\begin{rem}\label{open lemma}
\emph{Let $u$: $G\rightarrow H$ be a group homomorphism of topological groups with neutral elements $e_G$ and $e_H$. Then $u$ is open if and only if for every neighbourhood  $V \subseteq G$ of $e_G$, $u(V)$ is a neighbourhood of $e_H$ in $H$.}
\end{rem}

\begin{prop}\label{3->4}
$3) \Rightarrow 4)$: Let $M$ and $N$ be topological Hausdorff groups that have countable fundamental systems of open neighbourhoods of $0$. Assume that $M$ is complete. Consider a continuous group homomorphism $u$: $M\rightarrow N$ such that for all neighbourhoods $V$ of $0$ in $M$, $\overline{u(V)}$ is a neighbourhood of $0 $ in $ N$. Then $u$: $M\rightarrow N$ is open.
\end{prop}
\begin{proof}
 For each topological group $M$ and $N$ consider a metric that is right-invariant and induces its topology (c.f. (\ref{metrisable})). By hypothesis, $\overline{u(B_r(0))}$ is a neighbourhood of $0\in N \ \forall r>0$. So $\forall r>0$ there exists $\rho(r)>0$ such that $B_{\rho(r)}(0)\subseteq \overline{u(B_r(0))}$. Let $(x,r)$ be arbitrary in $M\times\mathbb{R}^{>0}$.
\begin{gather*}
\overline{u(B_r(x))}=\overline{u(B_r(0)\cdot x)}=\overline{u(B_r(0))\cdot u(x)}\\ = \overline{u(B_r(0))} \cdot u(x) \supseteq B_{\rho(r)}(0) \cdot u(x)= B_{\rho(r)}(u(x)).
\end{gather*}

\noindent So we can apply (\ref{closure of balls}) and particularly get $B_{\rho(r)}(0)\subseteq u(B_a(0)) \ \forall r<a$. This implies that for all neighbourhoods $V$ of $0$ in $M$, $u(V)$ is a neighbourhood of $0$ in $N$, and thus $u$ is open by (\ref{open lemma}).
\end{proof}

\subsection{4) implies 1)}

\begin{lem}\label{M=N}
Let $R$ be a topological ring that has a unit in each neighbourhood of $0$. Let M be a topological $R$-module, and let $N \subseteq M  $ be a submodule that contains a nonempty open subset $U$ of $M$.  Then one has $N=M$.
\end{lem}
\begin{proof}
By translation we can assume that $0\in U$.
Let $m \in M$ be arbitrary. By (\ref{conv.in.top.mod}) there is an $r\in R^{*}$ with $r\cdot m \ \in U \subseteq N$.  Since $N$ is closed under scalar multiplication, $m=r^{-1}\cdot r \cdot m$ is in $N$. Therefore we have $N=M$.
\end{proof}
\medskip
\begin{prop}\label{quotient complete}
Let $G$ be a metrisable group, and let $H \subseteq G$ be a closed subgroup. Then $G/H$ is metrisable. If $G$ is complete, then $G/H$ is complete.
\end{prop}
\begin{proof}
\ensuremath{} [BGT] IX, §3.1, Proposition 4 with the addition that $H\subseteq G$ is a closed subgroup if and only if $G/H$ is Hausdorff.
\end{proof} 
\medskip
\begin{prop}\label{4->1}
$4) \Rightarrow 1):$ Let $R$ be a topological ring that has a unit in each neighbourhood of $0$. Let $M$ and $N$ be topological Hausdorff $R$-modules such that $M$ is complete and has a countable fundamental system of open neighbourhoods of $0$. Let $u$: $M\rightarrow N$ be a linear, continuous and open map. Then $u$ is surjective, and $N$ has a countable fundamental system of open neighbourhoods of $0$ and is complete.
\end{prop}
\begin{proof}
The image of $u$ is an open submodule of $N$. Thus $u$ is surjective by (\ref{M=N}). Define $K:=\ker(u)$ and let $\pi$: $M\rightarrow M/K$ be the projection. $u$ induces an isomorphism of topological $R$-modules $\overline{u}$: $M/K \rightarrow N$ with $u=\overline{u}\circ \pi$. So the claim follows with (\ref{quotient complete}).
\end{proof} 

\subsection{Corollaries}

\noindent In this subsection we state two corollaries of the Open Mapping Theorem:

\begin{cor}\label{main.cor}
Consider the situation above. Assume that $N$ is complete, and that $u$: $M\rightarrow N$ is bijective. Then $u$ is an isomorphism of topological $R$-modules.
\end{cor}

\begin{theo}\label{lemma cont map}
Let $R$ be a complete topological Hausdorff ring that has a sequence of units converging to $0$ and a countable fundamental system of open neighbourhoods of $0$. Let $M$ and $N$ be topological $R$-modules such that $M$ is  finitely generated, Hausdorff, complete and has a countable fundamental system of open neighbourhoods of $0$. Then each linear map $f$: $M\rightarrow N$ is continuous.
\end{theo}
\begin{proof}
Let $\lbrace m_1, \ldots, m_n \rbrace$ be a generating system of the $R$-module $M$. Consider the linear surjection $\pi$: $R^n \rightarrow M, (r_1,\ldots,r_n)\mapsto \sum_{i=1}^n r_i\cdot m_i$. Then define $\bar{f}$: $R^n\rightarrow N, $ via $ \bar{f}:=f\circ \pi$. As scalar multiplication and addition are continuous in $M$ and $N$, $\pi$ and $\bar{f}$ are continuous. Applying Theorem (\ref{setting}) we see that $\pi$ is open. Let $U\subseteq N$ be open. Then
\begin{gather*}
f^{-1}(U)=\pi (\pi ^{-1}(f^{-1}(U)))=\pi(\bar{f}^{-1}(U))\subseteq M \ \text{is open}.
\end{gather*} 
\end{proof}

\section{Applications}

{As announced in the introduction, as from now all rings are \textbf{commutative}.\\

\noindent In this section we want to make use of the previous results. At first we will apply Theorem (\ref{setting}) to prove that over certain topological rings a complete module is noetherian if and only if each submodule is closed. We will use this later on to show that for such a ring $R$, if $R$ is noetherian, there exists a unique topological $R$-module structure on each finitely generated $R$-module $M$ such that $M$ is Hausdorff, complete and has a countable fundamental system of open neighbourhoods $0$ (i.e. $M$ is complete and metrisable). The proofs in this section are generalizations of arguments of [BGR]. \\

\subsection{Topologically nilpotent elements}

\begin{defn}\label{def top.nilpot.}
\medskip
\emph{Let $R$ be a topological ring. \begin{enumerate}
\item
$r\in R$ is said to be \emph{topologically nilpotent} if the sequence $(r^n)_{n\in \mathbb{N}}$ converges to $0$. Define $R^{oo}:=\lbrace r\in R \ | \ r \ \text{is topologically nilpotent}\rbrace $. 
\item $B\subseteq R$ is called \emph{bounded} if for every neighbourhood $U$ of $0$ in $R$ there is an open neighbourhood $V$ of $0$ in $R$ with $V\cdot B\subseteq U$.
\item $R$ is said to be \emph{non-archimedean} if $R$ has a fundamental system of neighbourhoods of $0$ consisting of subgroups of $(R,+)$.
\end{enumerate} 
}
\end{defn}
\medskip

\begin{rem}\label{endliche mengen beschränkt}
\emph{Let $R$ be a topological ring.
\begin{enumerate}
\item Each finite subset of $R$ is bounded.
\item For $a\in R^{oo}$ the set $X:=\lbrace a^n \ | \ n\in \mathbb{N}_0 \rbrace$ is bounded.
\item $R^{oo}$ is closed under ring multiplication.
\end{enumerate}}
\end{rem}

\medskip

\begin{lem}\label{R^{oo} is untergruppe}
Let $R$ be a non-archimedean ring. Then $R^{oo}$ is a subgroup of $(R,+)$.
\end{lem}

\begin{proof}
Let $a,b\in R^{oo}$ and $U$ be a neighbourhood of $0$ that is a subgroup of $(R,+)$. Since $\lbrace a^n \ | \ n\in \mathbb{N}_0 \rbrace$ and $\lbrace b^n \ | \ n\in \mathbb{N}_0 \rbrace$ are bounded, we can find open neighbourhoods $V_1$ and $V_2$ of $0$ with $a^{i} \cdot V_1\subseteq U$ and $b^{j}\cdot V_2\subseteq U$ for all $i,j\in \mathbb{N}_0$. Define $V:=V_1\cap V_2$. Since $a$ and $b$ are topologically nilpotent, we can find $N,M\in \mathbb{N}$ with $a^n, \ b^m\in V$ for all $n\geq N$, $m\geq M$. One has:
 \begin{gather*}
 (a+b)^{N+M}=  \sum_{j=0}^{N+M} {N+M \choose j} a^{N+M-j}\cdot b^j.
 \end{gather*}
 Since $U$ is additively closed, it is sufficient to show that $a^{k-j}\cdot b^{j}\in U$ for all $k\geq N+M$, $j=0,\ldots, k$. If $k-j\geq N$, then $a^{k-j}\cdot b^{j} \in V\cdot b^{j} \subseteq U$. Conversely, $k-j<N$ implies $j>M$. Hence $a^{k-j}\cdot b^{j} \in a^{k-j}\cdot V \subseteq U$.
\end{proof} 

\begin{rem}\label{properties of top nilpotent elements}
\emph{Let $R$ be a non-archimedean ring.
\begin{enumerate}
\item  $\sum_{i=0}^{\infty} r^{i}$ is a Cauchy sequence in $R$ for all $r\in R^{oo}$.
\end{enumerate}   
If $R$ is Hausdorff, complete and possesses a countable fundamental system of open neighbourhoods of $0$, one has:
\begin{enumerate}
\item[ii)] $1-r\in R^{\ast}$ for all $r\in R^{oo}$.
\end{enumerate}}
\end{rem}
\medskip
\begin{rem}\label{abzählbare umgebungsbasis überflüssig}
\emph{Note that in the situation above we assume a countable fundamental system of open neighbourhoods of $0$ to be able to work with the definition of complete topological groups as it is given in (\ref{complete}). If one defines completeness for topological groups in general, an analogous proposition holds as well. In any case we will need the countable fundamental system to apply Theorem (\ref{lemma cont map}) in the following propositions.}
\end{rem}
\medskip

\begin{lem}\label{lemma M=N+A => M=N}
Let $R$ be a complete non-archimedean topological Hausdorff ring that has a countable fundamental system of neighbourhoods of $0$. Let $M$ be an $R$-module, and let $N\subseteq M$ be a submodule such that there are $m_1,\ldots, m_l \in M$ with $M=N+\sum^{l}_{i=1} (R^{oo} \cdot m_i)$. Then one has $N=M$.
\end{lem}
\begin{proof}
 We want to show that $m_i \in N \ \forall \ i=1,\ldots ,l$. We can write 
\begin{gather*}
m_i=n_i+\sum^{l}_{j=1} a_{ij}m_j \  \  n_i \in N, a_{ij}\in R^{oo}; \forall \ i=1,\ldots ,l.
\end{gather*}
I.e. in matrix form:
\begin{gather*}
\begin{pmatrix} 
   m_1\\ 
   \vdots\\
   m_l 
\end{pmatrix}
=
\begin{pmatrix}
n_1\\
\vdots\\
n_l
\end{pmatrix}
+
\begin{pmatrix}
a_{11}& \ldots & a_{1l}\\
\vdots& \ddots & \vdots\\
a_{l1}& \ldots & a_{ll}
\end{pmatrix}
\cdot
\begin{pmatrix}
m_1\\
\vdots\\
m_l
\end{pmatrix}.
\end{gather*}
\begin{gather*}
\Leftrightarrow m=n+A\cdot m \Leftrightarrow n=(I_l-A)\cdot m
\end{gather*}
Since $N$ is closed under scalar multiplication and addition, it is sufficient to show that $(I_l-A)^{-1}$ exists, i.e. $det((I_l-A))\in R^{*}$. $R^{oo}$ is a subgroup of $(R,+)$ and closed under ring multiplication. Hence one checks by applying the Leibniz formula that $det((I_l-A))$ is of the form $1-a$ for $a\in R^{oo}$. Therefore $det((I_l-A))\in R^{*}$ by (\ref{properties of top nilpotent elements}).
\end{proof}
\medskip

\noindent Recall the following results on the completion of metric spaces:
\begin{rem}\label{completion of metric modules}
\emph{Let $R$ be a complete topological Hausdorff ring and $M$ be a metrisable topological $R$-module. There exists a topological $R$-module $M^{\textasciicircum}$ such that $M^{\textasciicircum}$ is metrisable, complete and contains $M$ as a dense subset (c.f. [BGT] III, §6.6).}
\end{rem}
\medskip
\begin{prop}\label{noetherian, completition finitely generated => complete}
Let $R$ be a complete non-archimedean topological Hausdorff ring that has a countable fundamental system of neighbourhoods of $0$. Assume that there is a zero sequence of units of $R$, and that $R^{oo}$ is a neighbourhood of $0$. Let $M$ be a metrisable topological $R$-module such that $M^{\textasciicircum}$ is a finitely generated $R$-module. Then one has $M=M^{\textasciicircum}$.
\end{prop}
\begin{proof}
Let $\pi$: $R^n\rightarrow M^{\textasciicircum}, (r_1,\ldots,r_n)\mapsto \sum_{i=1}^{n}(r_i \cdot m_i) $ ($m_i \in M^{\textasciicircum	}$) be surjective . $\pi$ is continuous, and hence it is open by (\ref{setting}). Particularly, $\sum_{i=1}^{n}(R^{oo} \cdot m_i)$ is a neighbourhood of $0$ in $M^{\textasciicircum}$. Since in an arbitrary topological group the product of a dense subset and any neighbourhood is the whole group, one has:
\begin{gather*}
M^{\textasciicircum}=M+\sum_{i=1}^{n}(R^{oo} \cdot m_i).
\end{gather*}
No we can apply (\ref{lemma M=N+A => M=N}) and get $M=M^{\textasciicircum}$ .
\end{proof}
\medskip

\begin{defn}\label{defn. Tate Ring}
\emph{A topological ring $R$ is called a \emph{Tate ring} if it satisfies both of the following properties:
\begin{enumerate}
\item $R$ has a topologically nilpotent unit.
\item There exists an open subring $R_0$ of $R$ and a finitely generated ideal $I$ of $R_0$ such that $(I^n)_{n\in \mathbb{N}}$ is a fundamental system of neighbourhoods of $R_0$.
\end{enumerate}
}
\end{defn}

\begin{rem}\label{properties of a tate ring}
\emph{ In the proposition above and the following propositions we considered and will consider a complete topological Hausdorff ring $R$ that has the following properties:
\begin{enumerate}
\item $R$ has a countable fundamental system of open neighbourhoods of $0$.
\item $R$ is non-archimedean.
\item $R^{oo}$ is a neighbourhood of $0$.
\item $R$ has a zero sequence consisting of units .
\end{enumerate}
Note that $R$ for instance satisfies  i)- iv) if it is a complete topological Hausdorff Tate ring.}
\end{rem}

\medskip
\begin{prop}\label{noeth equival closed}
Let $R$ be a complete non-archimedean topological Hausdorff ring that has a countable fundamental system of neighbourhoods of $0$. Assume that there is a zero sequence of units of $R$, and assume that $R^{oo}$ is a neighbourhood of $0$. Consider a complete topological Hausdorff $R$-module $M$ that has a countable fundamental system of open neighbourhoods of $0$. Then one has:
\begin{enumerate}
\item $M$ is noetherian if and only if every submodule of $M$ is closed.
\item  $R$ is noetherian if and only if every ideal of $R$ is closed.
\end{enumerate}
\end{prop}
\begin{proof}
$ii)$ follows from $i)$.\\ $i)$ $\glqq\Rightarrow\grqq:$  Since $M$ is complete and noetherian, for all submodules $N$ of $M$, $N^{\textasciicircum}$ is a finitely generated submodule of $M$. By (\ref{noetherian, completition finitely generated => complete}) we see that $N$ is already complete. Hence it is closed in $M$.\\
$\glqq\Leftarrow\grqq:$ Let all submodules of $M$ be closed and consider an ascending chain of such:
\begin{equation*}
M_0\subseteq M_1 \subseteq M_2 \subseteq \ldots \ .
\end{equation*}
 $M^{*}:=\bigcup_{n\in \mathbb{N}_0} M_n$ is a submodule of $M$ and therefore a Baire space. Hence there is an $i\in \mathbb{N}_{0}$ such that $M_i$ possesses an interior point. But then one has $M_i=M^{*}$ by (\ref{M=N}). Therefore the considered chain becomes stationary. 
\end{proof} 
\medskip

\subsection{Unique complete and metrisable topology}

\medskip
\begin{theo}\label{unique r mod strc}
Let $R$ be a complete non-archimedean noetherian topological Hausdorff ring that has a countable fundamental system of neighbourhoods of $0$. Assume that there is a zero sequence of units of $R$ and that $R^{oo}$ is a neighbourhood of $0$ (e.g. $R$ is a complete noetherian Hausdorff Tate ring).
Let $M$ be a finitely generated $R$-module.\\ Then there exists a unique $R$-module topology on $M$ such that $M$ is Hausdorff, complete and has a countable fundamental system of open neighbourhoods of $0$ (i.e. $M$ is metrisable and complete).
\end{theo}
\noindent 
The uniqueness in the theorem above holds as well if $R$ is a complete topological Hausdorff ring that has a sequence of units converging to $0$ and a countable fundamental system of open neighbourhoods of $0$. Note that for a noetherian ring $R$ the $R$-module $R^n$ is noetherian for all $n\in \mathbb{N}$.\\

\begin{proof}
\textit{Uniqueness:} Let $\mathcal{T}$ and $\mathcal{T}'$ be two topologies on $M$ such that $M$ with each of this topologies is a complete topological Hausdorff $R$-module that has a countable fundamental system of open neighbourhoods of $0$. Define $M:=(M,\mathcal{T})$ and $M':=(M,\mathcal{T}')$. Consider the identity $id$: $M\rightarrow M'$, which is linear and bijective. By (\ref{lemma cont map}) $id$ is continuous and hence an isomorphism of topological $R$-modules by (\ref{main.cor}). So we have $\mathcal{T}=\mathcal{T}'$.\\
\noindent \textit{Existence:} Consider a linear surjection $\pi$: $R^n \rightarrow M$. Since $R^n$ is noetherian, $K:=\ker(\pi)$ is closed in $R^n$ by (\ref{noeth equival closed}). As $R^n$ is complete and metrisable, $R^n/K$ is complete and metrisable by (\ref{quotient complete}). Therefore it is Hausdorff and has a countable fundamental system of open neighbourhoods of $0$ by (\ref{metrisable}). Let $\bar{\pi}$: $R^n/K\rightarrow M$ be the induced linear bijection. Endow $M$ with the topology such that $U\subseteq M$ is open if and only if $\bar{\pi}^{-1}(U)\subseteq R^n/K$ is open. Then $\bar{\pi}$ is an isomorphism of topological $R$-modules and hence there exists a desired topology on $M$.
\end{proof}


\begin{thebibliography}{3}

\bibitem[BGT]{BGT} N. Bourbaki; \emph{General Topology}, chap. I - IV, Springer (1989); chap. \ \ V - X, Springer (1989).
\bibitem[BGR]{} S. Bosch, U. Güntzer, R. Remmert; \emph{Non-Archimedean Analysis}, Springer (1984).
\bibitem[BTVS]{} N. Bourbaki; \emph{Topological Vector Spaces}, chap. I - V, Springer (1987).	

\end{thebibliography}
\end{document}